\documentclass[a4paper,12pt]{article}     
\usepackage{amsmath,amscd,amssymb}        
\usepackage{latexsym}                     
\usepackage[english, german]{babel}                
\usepackage{xypic}

\usepackage{theorem}                 
\usepackage{color}

\input{cyracc.def}

\newfont{\cyrfnt}{wncyr10}

\newtheorem{theorem}{Theorem}

\newtheorem{lemma}{Lemma}

\theorembodyfont{\rmfamily}
\newtheorem{remark}{Remark}

\newenvironment{definition}
{\smallskip\noindent{\bf Definition\/}:}{\smallskip\par}

\newenvironment{examples}
{\smallskip\noindent{\bf Examples\/}.}{\smallskip\par}

\newenvironment{proof}{\begin{ProofwCaption}{Proof}}{\end{ProofwCaption}}
\newenvironment{proof*}[1]{\begin{ProofwCaption}{{#1}}}{\end{ProofwCaption}}
\newenvironment{ProofwCaption}[1]%
  {\addvspace\theorempreskipamount \noindent{\it #1.}\rm}%
  {\qed \par \addvspace\theorempostskipamount}
\newcommand{\qedsymbol}{{\rm $\Box$}}
\newcommand{\qed}{\hfill\qedsymbol}

\newcommand{\CC}{{\mathbb C}}

\newcommand{\ZZ}{{\mathbb Z}}

\newcommand{\calG}{{\mathcal G}}

\newcommand{\Ind}{{\rm Ind}\,}

\newcommand{\chiorb}{{\chi}^{\rm orb}}
\newcommand{\chiorbbar}{{\overline{\chi}}^{\,\rm orb}}

\title{On the orbifold Euler characteristics of dual invertible polynomials with non-abelian
symmetry groups}
\author{Wolfgang Ebeling and Sabir M.~Gusein-Zade
\thanks{Partially supported by DFG. The work of the second author
(Sections~\ref{sect:Intro}, \ref{sect:NC-Saito}, \ref{sect:abelian_Euler})
was supported by the grant 16-11-10018 of the Russian Science Foundation.
}
}
\date{}

\begin{document}
\selectlanguage{english}

\maketitle

\begin{abstract}
In the framework of constructing mirror symmetric pairs of Calabi-Yau manifolds, P.~Berglund, T.~H\"ubsch
and M.~Henningson considered a pair $(f,G)$ consisting of an invertible polynomial $f$ and a finite abelian
group $G$ of its diagonal symmetries and associated to this pair a dual pair $(\widetilde{f}, \widetilde{G})$.
A.~Takahashi suggested a generalization of this construction to pairs $(f, G)$ where $G$ is a non-abelian group generated by some diagonal symmetries and some permutations of variables. In a previous paper, the authors showed that some mirror
symmetry phenomena appear only under a special condition on the action of the group $G$: a parity condition.
Here we consider the orbifold Euler characteristic of the Milnor fibre of a pair $(f,G)$. We show that,
for an abelian group $G$, the mirror symmetry of the orbifold Euler characteristics can be derived from
the corresponding result about the equivariant Euler characteristics. For non-abelian symmetry groups we
show that the orbifold Euler characteristics of certain extremal orbit spaces of the group $G$
and the dual group $\widetilde{G}$ coincide. From this we derive that the orbifold Euler
characteristics of the Milnor fibres of dual periodic loop polynomials coincide up to sign.
\end{abstract}

\section{Introduction} \label{sect:Intro}
The famous method of P.~Berglund and T.~H\"ubsch \cite{BH1} associates to a so-called invertible polynomial its
transpose polynomial. More generally, Berglund and M.~Henningson \cite{BH2} considered a pair $(f,G)$ consisting
of an invertible polynomial $f$ and a finite abelian group $G$ of diagonal symmetries of $f$. They associate
to this pair a dual pair $(\widetilde{f}, \widetilde{G})$ which we call the Berglund-H\"ubsch-Henningson (BHH)
dual pair. The purpose of their construction was to obtain mirror symmetric pairs of Calabi-Yau manifolds.
There were discovered mirror symmetry phenomena concerning the elliptic genera of BHH dual pairs
\cite{BH2,Kawai-Yang}. In a series of papers \cite{EG-BLMS, EG-MMJ, EG-PEMS} and partly in joint work
with A.~Takahashi \cite{EGT}, we started to look at these pairs from a singularity theory point of view.
In \cite{EG-BLMS}, we studied the equivariant Euler characteristic of the Milnor fibre of $f$ with the action of $G$
defined as an element of the Burnside ring of the group. We constructed a duality between the Burnside rings
of a finite abelian group and of its group of characters which generalizes the Saito duality and we showed
that this duality is related to the BHH duality. 
Namely, it was shown that the reduced equivariant Euler
characteristics of the Milnor fibres of dual pairs are dual to each other up to sign. In \cite{EG-MMJ},
we studied the reduced orbifold Euler characteristic of the Milnor fibre of such a pair and we showed
that for dual pairs these invariants coincide up to sign.
(In some papers, symmetry properties of BHH-dual pairs were considered under the
 condition that the symmetry group $G$ is a so-called admissible group. In \cite{EG-BLMS, EG-MMJ, EG-PEMS, EGT},
 no conditions on $G$ are imposed.)

Recently, there has been some interest in generalizing the construction to non-abelian groups of symmetries
(cf.\ \cite{FJR2015}). Following an idea of Takahashi, we considered a semi-direct product $G \rtimes S$ where
$G$ is a subgroup of the group $G_f$ of all diagonal symmetries of $f$ and $S$ is a subgroup of the group $S_n$
of permutations of the variables preserving $f$ and respecting $G$. Takahashi proposed a natural candidate for
the group dual to $G \rtimes S$. With this construction, we generalized the Saito duality between Burnside rings
to this case of non-abelian groups and proved a ``non-abelian'' generalization of the statement about equivariant
Euler characteristics \cite{EG-NC}. It turned out that the statement only holds under a special condition
on the action of the subgroup $S$ of the permutation group called PC (``parity condition''). Moreover,
it turned out that the pairs from a collection given in \cite{Yu} dual in the sense of Takahashi
but not satisfying the PC condition are not mirror symmetric.

Here we consider the orbifold Euler characteristic of the Milnor fibre of a pair $(f,G)$. We show that,
for an abelian group $G$, the result of \cite{EG-MMJ} about the orbifold Euler characteristics can be derived
from the corresponding result about the equivariant Euler characteristics. Then we consider a non-abelian
group of the above form. Let $S$ be a subgroup of $S_n$ satisfying PC and let $T \subset S$ be a subgroup of $S$.
As the main result of this paper, we derive that the orbifold Euler characteristics of certain extremal orbit spaces
of the group $G\rtimes S$ and the dual group $\widetilde{G\rtimes S}$ coincide
(see Theorem~\ref{thm:main}). We derive from this that the orbifold Euler characteristics 
of the Milnor fibres of dual periodic loop polynomials coincide up to sign.

\section{Invertible polynomials and their symmetry groups} \label{sect:invertible_poly}
An invertible polynomial in $n$ variables is a quasihomogeneous polynomial $f$ with the number of monomials 
equal to the number $n$ of variables (that is 
\[ f(x_1, \ldots , x_n)=\sum_{i=1}^n a_i \prod_{j=1}^n x_j^{E_{ij}} \, ,
\]
where $a_i$ are non-zero complex numbers and $E_{ij}$ are non-negative integers) such that the matrix
$E=(E_{ij})$ is non-degenerate and $f$ has an isolated critical point at the origin. 

\begin{remark}
The condition $\det{E}\ne 0$ is equivalent to the condition that the weights $q_1$, \dots, $q_n$ of
the variables in the polynomial $f$ are well defined
(if one assumes the quasidegree to be equal to $1$). In fact they are defined by the equation
$$
E\cdot(q_1, \ldots, q_n)^T=(1,\ldots, 1)^T\,.
$$
Without loss of generality one may assume that all the coefficients $a_i$ are equal to $1$.
\end{remark}

A classification of invertible polynomials is given in \cite{KS}. Each invertible polynomial is the direct
(``Sebastiani--Thom'') sum of atomic polynomials in different sets of variables of the following types:
\begin{enumerate}
 \item[1)] chains: $x_1^{p_1}x_2+x_2^{p_2}x_3+\ldots+x_{m-1}^{p_{m-1}}x_m+x_m^{p_m}$,
 $m\ge 1$\,;
 \item[2)] loops: $x_1^{p_1}x_2+x_2^{p_2}x_3+\ldots+x_{m-1}^{p_{m-1}}x_m+x_m^{p_m}x_1$,
 $m\ge 2$\,.
\end{enumerate}

The group of the diagonal symmetries of $f$ is
$$
G_f=\{\underline{\lambda}=(\lambda_1, \ldots, \lambda_n)\in(\CC^*)^n:
f(\lambda_1x_1, \ldots, \lambda_nx_n)=f(x_1, \ldots, x_n)\}\,.
$$
One can see that $G_f$ is an abelian group of order $\vert\det E\,\vert$.
For an atomic polynomial the group $G_f$ is cyclic.
The Milnor fibre $V_f=\{x\in\CC^n:f(x)=1\}$ of the invertible polynomial $f$
is a complex manifold of dimension $n-1$ with the natural action of the group $G_f$. 

The Berglund--H\"ubsch transpose of $f$ is
$$
\widetilde{f}(x_1, \ldots, x_n)=\sum_{i=1}^n \prod_{j=1}^n x_j^{E_{ji}}\,.
$$
The group $G_{\widetilde{f}}$ of the diagonal symmetries of $\widetilde{f}$ is in a canonical way isomorphic
to the group $G_f^*={\rm Hom}(G_f,\CC^*)$ of characters of $G_f$ (see, e.~g., \cite[Proposition 2]{EG-BLMS}).
For a subgroup $G$ of $G_f$, the {\em Berglund--H\"ubsch--Henningson} (BHH) {\em dual} to the pair $(f,G)$ is the pair
$(\widetilde{f},\widetilde{G})$, where $\widetilde{G}\subset G_{\widetilde{f}}=G_f^*$ is the subgroup of
characters of $G_f$ vanishing (i.e.\ being equal to $1$) on the subgroup $G$.

Let the permutation group $S_n$ act on the space $\CC^n$ by permuting the variables. If an invertible polynomial
$f$ is invariant with respect to a subgroup $S\subset S_n$, then it is invariant with respect to the semidirect
product $G_f\rtimes S$ (defined by the natural action of $S$ on $G_f$). The Milnor fibre $V_f$
of the polynomial $f$ carries an action of the group $G_f\rtimes S$.

Let $G$ be a subgroup of $G_f$ invariant with respect to the group $S$. In this case the semidirect product
$G\rtimes S$ is defined and the BHH dual subgroup $\widetilde G$ is also invariant with respect to $S$.
An idea to define a pair dual to $(f,G\rtimes S)$ was suggested by A.Takahashi.

\begin{definition}
 The {\em Berglund--H\"ubsch--Henningson--Takahashi} (BHHT) {\em dual} to the pair $(f,G\rtimes S)$ is the pair
$(\widetilde{f},\widetilde{G}\rtimes S)$.
\end{definition}

The Burnside ring $A(H)$ of a finite group $H$ is the Grothendieck ring of finite $H$-sets: see, e.~g.,
\cite{Knutson}. As an abelian group, $A(H)$ is freely generated by the classes $[H/K]$ of the quotient sets $H/K$
for representatives $K$ of the conjugacy classes of subgroups of $H$. For an $H$-space $X$ and for a point $x\in X$
the isotropy subgroup of $x$ is $H_x :=\{g\in H: gx=x\}$. For a subgroup $K\subset H$ the set of fixed points of $K$
(that is points $x$ with $H_x\subset K$) is denoted by $X^K$; the set of points $x\in X$ with the isotropy subgroup
$K$ is denoted by $X^{(K)}$, the set of points $x\in X$ with the isotropy subgroup conjugate to $K$ is denoted by
$X^{([K])}$. 
For a ``sufficiently nice'' topological space $Z$, denote by $\chi(Z)$ its (additive)
Euler characteristic, i.e.\ the alternating sum of the ranks of the cohomology groups
with compact support.
The equivariant Euler characteristic of a topological $H$-space $X$ is
the element of the Burnside ring $A(H)$ defined by
\begin{equation}\label{eq:equiv_chi}
 \chi^H(X) :=\sum_{[K]\in{\rm Conjsub\,}H}\chi(X^{([K])}/H)[H/K]\,,
\end{equation}
where ${\rm Conjsub\,}H$ is the set of the conjugacy classes of subgroups of $H$.
The reduced equivariant Euler characteristic $\overline{\chi}^H(X)$ is $\chi^H(X)-\chi^H(pt)=\chi^H(X)-[H/H]$.

The {\em orbifold Euler characteristic} of the pair $(X, H)$ is defined by 
\begin{equation} 
\chiorb(X,H)=\frac{1}{\vert H\vert}\sum_{(g,h)\in H^2:gh=hg} \chi(X^{\langle g,h\rangle})\,,
\end{equation}
where $\langle g,h\rangle$ is the subgroup of $H$ generated by $g$ and $h$ (see \cite{AS, HH} and references therein).
The {\em reduced orbifold Euler characteristic} of an $H$-set $X$ is defined as 
${\chiorbbar}(X,H)= \chiorb(X,H)- \chiorb(pt,H)$, where $\chiorb(pt,H)=\vert {\rm Conj\,}H \vert$
is the orbifold Euler characteristic of a one-point set with the only $H$-action.
For an abelian group $H$ one has $\chiorb(pt,H)=\vert H\vert$.

Since each element of the Burnside ring $A(H)$ is represented by a zero-dimensional space with an $H$-action,
its orbifold Euler characteristic is defined. This defines a group (not a ring!) homomorphism $\chiorb:A(H)\to\ZZ$.
Moreover, for an $H$-space $X$ one has $\chiorb(X,H)=\chiorb(\chi^H(X))$.

If $H$ is a subgroup of a finite group $G$, one has the {\em reduction} and the {\em induction} operations
${\rm Red}^G_H$ and ${\rm Ind}^G_H$ which convert $G$-spaces to $H$-spaces and $H$-spaces to $G$-spaces
respectively. The reduction ${\rm Red}^G_HX$ of a $G$-space $X$ is the same space considered with the action
of the smaller subgroup. (One can say that ${\rm Red}^G_HX$ converts a pair $(X,G)$ to the pair $(X,H)$.)
The induction ${\rm Ind}^G_HX$ of an $H$-space $X$ is the quotient space $(G\times X)/\sim$,
where the equivalence relation $\sim$ is defined by: $(g_1,x_1)\sim(g_2,x_2)$ if (and only if) there exists $h\in H$
such that $g_2=g_1h$, $x_2=h^{-1}x_1$; the $G$-action on it is defined in the natural way. Applying the reduction and
the induction operations to finite $G$- and $H$-sets respectively, one gets the {\em reduction homomorphism}
${\rm Red}^G_H:A(G)\to A(H)$ and the {\em induction homomorphism} ${\rm Ind}^G_H:A(H)\to A(G)$. For a subgroup $K$
of $H$, one has ${\rm Ind}^G_H[H/K]=[G/K]$. The reduction homomorphism is a ring homomorphism, whereas the
induction one is a homomorphism of abelian groups.

For a finite abelian group $H$, let $H^*={\rm Hom}(H,\CC^*)$ be its group of characters. Just as for a subgroup
of $G_f$ above, for a subgroup $K\subset H$, the (BHH) dual subgroup of $H^*$ is
$$
\widetilde{K} :=\{\alpha\in H^*: \alpha(g)=1 \text{ for all } g\in K\}\,.
$$
The equivariant Saito duality (see \cite{EG-BLMS}) is the group homomorphism $D_H:A(H)\to A(H^*)$ defined by
$D_H([H/K]) :=[H^*/\widetilde{K}]$. In \cite{EG-BLMS}, it was shown that 
\begin{equation}\label{eq:Saito_dual}
 \overline{\chi}^{G_f}(V_f) = (-1)^n D_{G_{\widetilde{f}}} \overline{\chi}^{G_{\widetilde{f}}}(V_{\widetilde{f}})\,,
\end{equation}
i.e.\ the reduced equivariant Euler characteristics
of the Milnor fibres of Berglund--H\"ubsch dual invertible polynomials $f$ and $\widetilde{f}$ with the actions
of the groups $G_f$ and $G_{\widetilde{f}}$ respectively are Saito dual to each other up to the sign $(-1)^n$.

\section{Non-abelian equivariant Saito duality}\label{sect:NC-Saito}
Let $G$ be a finite abelian group and let $S$ be a finite group with a homomorphism $\varphi:S\to {\text{Aut\,}}G$.
These data determine the semi-direct product $\widehat{G}=G\rtimes S$. Let $A^{\rtimes}(G\rtimes S)$ be the
Grothendieck group of finite $\widehat{G}$-sets with the isotropy subgroups of points conjugate to
$H\rtimes T\subset G\rtimes S$, where $H$ and $T$ are subgroups of $G$ and of $S$ respectively such that,
for $\sigma\in T$, the automorphism $\varphi(\sigma)$ preserves $H$. (The semidirect product structure on $H\rtimes T$
is defined by the homomorphism $\varphi_{\vert T}:T\to {\text{Aut\,}}H$.) The group $A^{\rtimes}(G\rtimes S)$ is
a subgroup of the Burnside ring $A(\widehat{G})$ of the group $\widehat{G}$.
It is the free abelian group generated by the conjugacy classes of the subgroups of the form $H\rtimes T$.
An element of $A^{\rtimes}(G\rtimes S)$ can be written in a unique way as 
$$
\sum_{[H\rtimes T]\in {\text{Conjsub\, }}\widehat{G}} a_{H\rtimes T}\left[G\rtimes S/H\rtimes T\right]
$$
with integers $a_{H\rtimes T}$.

Let $G^*={\text{Hom\, }}(G,\CC^*)$ be the group of characters on $G$. One has $G^{**}\cong G$ (canonically).
The homomorphism $\varphi:S\to {\text{Aut\,}}G$ induces a natural homomorphism $\varphi^*:S\to {\text{Aut\,}}G^*$:
$\langle\varphi^*(\sigma)\alpha, g\rangle=\langle\alpha, \varphi(\sigma^{-1})g\rangle$, where
$\langle\alpha, g\rangle:=\alpha(g)$. Let $\widehat{G}^*:=G^*\rtimes S$ be the semidirect product defined by the
homomorphism $\varphi^*$.
One can see that, if $\varphi(\sigma)$ preserves a subgroup $H\subset G$, then $\varphi^*(\sigma)$
preserves the subgroup $\widetilde{H}\subset G^*$. Thus for a semidirect product $H\rtimes T\subset G\rtimes S$
one has the semidirect product $\widetilde{H}\rtimes T\subset G^*\rtimes S$.

One can show that
 subgroups $H_1\rtimes T_1$ and $H_2\rtimes T_2$ are conjugate in $G\rtimes S$ if and only if the subgroups
 $\widetilde{H_1}\rtimes T_1$ and $\widetilde{H_2}\rtimes T_2$ are conjugate in $G^*\rtimes S$ (see \cite[Proposition~2]{EG-NC}).
Therefore the following definition makes sense.

\begin{definition}
 The {\em (``non-abelian'') equivariant Saito duality} corresponding to the group $\widehat{G}=G\rtimes S$
 is the group homomorphism $D^{\rtimes}_{\widehat{G}}:A^{\rtimes}(G\rtimes S)\to A^{\rtimes}(G^*\rtimes S)$
 defined (on the generators) by
 $$
 D^{\rtimes}_{\widehat{G}}([G\rtimes S/H\rtimes T])=
 [G^*\rtimes S/\widetilde{H}\rtimes T]\,.
 $$
\end{definition}

One can see that $D^{\rtimes}_{\widehat{G}}$ is an isomorphism of the groups $A^{\rtimes}(G\rtimes S)$ and
$A^{\rtimes}(G^*\rtimes S)$ and $D^{\rtimes}_{\widehat{G^*}}D^{\rtimes}_{\widehat{G}}={\rm id}$\,.

For a subgroup $S'\subset S$ one has the natural homomorphism
$\Ind_{G\rtimes S'}^{G\rtimes S}:A^{\rtimes}(G\rtimes S')\to A^{\rtimes}(G\rtimes S)$ sending 
the generator $\left[G\rtimes S'/H\rtimes T\right]$ to the generator $\left[G\rtimes S/H\rtimes T\right]$.
This homomorphism commutes with the Saito duality, i.e.\ the diagram
$$
\begin{CD}
A^{\rtimes}(G\rtimes S') @>D^{\rtimes}_{G\rtimes S'}>> A^{\rtimes}(G^*\rtimes S')\\
@VV\Ind_{G\rtimes S'}^{G\rtimes S}V
@VV\Ind_{G^*\rtimes S'}^{G^*\rtimes S}V
\\
A^{\rtimes}(G\rtimes S) @>D^{\rtimes}_{G\rtimes S}>> A^{\rtimes}(G^*\rtimes S)
\end{CD}
$$
is commutative.

Let $f$ be an invertible polynomial invariant with respect to a subgroup $S \subset S_n$.

\begin{definition}
 We say that a subgroup $S\subset S_n$ satisfies the {\em parity condition} (``PC'' for short) if for each subgroup
 $T\subset S$ one has
 $$
 \dim (\CC^n)^T\equiv n\ \ \mod 2\,,
 $$
 where $(\CC^n)^T=\{\underline{x}\in \CC^n: \sigma \underline{x}= \underline{x} \text{ for all } \sigma\in T\}$.
\end{definition}

\begin{examples}
 {\bf 1.} A subgroup $S\subset S_n$ satisfying PC is contained in the alternating group
 $A_n\subset S_n$. A cyclic subgroup of $S_n$ satisfies PC if and only if it is contained in $A_n$.

{\bf 2.} For $n\ge 4$, the subgroup $A_n\subset S_n$ does not satisfy PC.

{\bf 3.} The subgroup $S=\langle (12)(34), (13)(24) \rangle\subset A_4$ isomorphic to
$\ZZ_2 \times \ZZ_2$ does not satisfy PC. The group $\langle (12345), (12)(34) \rangle \subset A_5$ coincides
with $A_5$ and therefore does not satisfy PC. The group $\langle (12345), (14)(23) \rangle \subset A_5$
is isomorphic to the dihedral group $D_{10}$ and satisfies PC.
\end{examples}

Let $V_f$ denote the Milnor fibre of $f$. From \cite[Proposition~3]{EG-NC} one can derive 
that the equivariant Euler characteristic of the Milnor fibre $V_f$ of a polynomial $f$ with
the $G_f \rtimes S$-action belongs to the subgroup $A^{\rtimes}(G_f \rtimes S) \subset A(G_f \rtimes S)$.
Let
${\widehat{G}}_f=G_f\rtimes S$, ${\widehat{G}}_{\widetilde{f}}=G_{\widetilde{f}}\rtimes S$.
It was proved in \cite[Theorem~1]{EG-NC} that, if the subgroup $S\subset S_n$ satisfies PC, then one has
 \begin{equation}\label{eq:NC_Saito}
   \overline{\chi}^{G_f \rtimes S}(V_{f})
  =(-1)^n D_{G_{\widetilde{f}} \rtimes S}^{\rtimes} \overline{\chi}^{G_{\widetilde{f}} \rtimes S}(V_{\widetilde{f}})\,. 
 \end{equation}

\section{Orbifold Euler characteristic: the abelian case}\label{sect:abelian_Euler}
Let $f$ be an invertible polynomial in $n$ variables, $G$ be a subgroup of $G_f$, and $(\widetilde{f},\widetilde{G})$ the BHH-dual pair.

In \cite{EG-MMJ} it was shown that the reduced orbifold Euler characteristics of the Milnor fibres of
$(f, G)$ and $(\widetilde{f},\widetilde{G})$ 
(that is of the Milnor fibres
of the polynomials $f$ and $\widetilde{f}$ with the actions of the groups $G$ and $\widetilde{G}$ respectively)
coincide up to the sign $(-1)^n$. At that moment we did not realize relations between the equivariant Euler
characteristic and the orbifold one. (Moreover, taking into account the fact that the monodromy transformation of
an invertible polynomial is an element of the symmetry group $G_f$, we considered $\chi^G(V_f)$ rather not as
a generalization of the Euler characteristic, but as an equivariant version of the monodromy zeta function (and
even did not denote it by $\chi^G(V_f)$, but by $\zeta_f^{G_f}$).) Therefore the proof of \cite[Theorem~1]{EG-MMJ}
was independent of the result of \cite{EG-BLMS} (though the ideas elaborated in the proof of the result of
\cite{EG-BLMS} were used in \cite{EG-MMJ} as well).
Later we understood that the result of \cite{EG-MMJ} follows directly from \cite[Theorem~1]{EG-BLMS} due to the
following statement. 
Let $\calG$ be a finite abelian group with the group of characters $\calG^*={\rm Hom}(\calG,\CC^*)$,
let $G$ be a subgroup of $\calG$ and let $\widetilde{G}$ be the BHH-dual subgroup of $\calG^*$
(i.e.\ $\widetilde{G}=\{\alpha\in \calG^*: \alpha(g)=1 \text{\ \ for all\ \ }g\in G\}$).

\begin{theorem} \label{theo:abelian}
One has
\begin{equation}\label{equi2orb}
\chiorb\circ {\rm Red}^\calG_G = \chiorb\circ {\rm Red}^{\calG^*}_{\widetilde{G}}\circ D_{\calG^*}\,.
\end{equation}
(Both sides of~(\ref{equi2orb}) are group homomorphisms from the Burnside ring $A(\calG)$ to $\ZZ$.)
\end{theorem}

\begin{proof}
It is sufficient to verify~(\ref{equi2orb}) on the generators $[\calG/K]$ of $A(\calG)$. One has
\begin{eqnarray*}
\chiorb\circ {\rm Red}^\calG_G([\calG/K]) &= & \chiorb(\calG/K,G)\\
&=&\frac{\vert \calG\vert}{\vert K+G\vert}\cdot\chiorb((K+G)/K,G)=
\frac{\vert \calG\vert}{\vert K+G\vert}\cdot\vert K\cap G\vert;
\end{eqnarray*}
\begin{eqnarray*}
 \chiorb\circ {\rm Red}^{\calG^*}_{\widetilde{G}}\circ D_{\calG^*}([\calG/K])&=&
 \chiorb\circ {\rm Red}^{\calG^*}_{\widetilde{G}}([\calG^*/\widetilde{K}])=
 \frac{\vert \calG^*\vert}{\vert\widetilde{K}+\widetilde{G}\vert}\vert\cdot\vert \widetilde{K}\cap \widetilde{G}\vert\\
 \ &=&\frac{\vert \calG^*\vert}{\vert\widetilde{K\cap G}\vert}\vert\cdot\vert \widetilde{K+G}\vert=
 \vert K\cap G\vert\cdot \frac{\vert \calG\vert}{\vert K+G\vert}\,.
\end{eqnarray*}
\end{proof}

Applied to (\ref{eq:Saito_dual}), Theorem~\ref{theo:abelian} with $\calG=G_f$, $\calG^*=G_{\widetilde{f}}$
implies the following statement.

\begin{theorem}[\cite{EG-MMJ}]\label{theo:Euler_abelian}
\begin{equation}
\chiorbbar(V_f,G) = (-1)^n\, \chiorbbar(V_{\widetilde{f}}, \widetilde{G}).
\end{equation}
\end{theorem}

\section{Orbifold Euler characteristic: the non-abelian case}\label{sect:non-abelian_Euler}
Let $S$ be a subgroup of $S_n$, 
let $f$ be an invertible polynomial invariant with respect to $S$,
and let $G$ be a subgroup of $G_f$ preserved by $S$.
We do not know whether the natural analogue of Theorem~\ref{theo:abelian} holds 
for the group $\calG=G_f \rtimes S$, i.e., whether, for a subgroup $T$ of $S$, one has
\begin{equation}\label{equi2orb_non}
\chiorb\circ {\rm Red}^{G_f \rtimes S}_{G\rtimes T} = 
\chiorb\circ {\rm Red}^{G_{\widetilde{f}} \rtimes S}_{\widetilde{G}\rtimes T}\circ D_{G_{\widetilde{f}} \rtimes S}\,.
\end{equation}
(It seems that the condition PC is not important here.) Therefore, for a subgroup $S$ satisfying PC,
we cannot deduce an analogue of Theorem~\ref{theo:Euler_abelian} from Equation~(\ref{eq:NC_Saito}).
Equation~(\ref{equi2orb_non}) holds if and only if
 \begin{equation*}
  \chi^{\rm orb}(G_f\rtimes S/H\rtimes T, G\rtimes S)=
  \chi^{\rm orb}(G_{\widetilde{f}}\rtimes S/\widetilde{H}\rtimes T,
  \widetilde{G}\rtimes S)\,.
 \end{equation*}
We can only prove the following special case.

\begin{theorem} \label{thm:main}
Let $f$ be an invertible polynomial invariant with respect to a subgroup $S$ of $S_n$, let $G$ be a subgroup
of $G_f$ preserved by $S$, and let $T$ be a subgroup of $S$. Then one has
 \begin{equation}\label{eq:main}
  \chi^{\rm orb}(G_f\rtimes S/\{e\}\rtimes T, G\rtimes S)=
  \chi^{\rm orb}(G_{\widetilde{f}}\rtimes S/G_{\widetilde{f}}\rtimes T, 
  \widetilde{G}\rtimes S).
 \end{equation}
\end{theorem}

\begin{proof}
 Let us compute the left hand side of (\ref{eq:main}). We have
 \begin{eqnarray*}
 \lefteqn{\chi^{\rm orb}(G_f\rtimes S/\{e\}\rtimes T, G\rtimes S)} \\
 &=&\frac{1}{\vert G\rtimes S\vert}
 \sum_{{((g,\sigma),(g',\sigma'))\in \widehat{G}^2:}\atop{(g,\sigma)(g',\sigma')=(g',\sigma')(g,\sigma)}}
 \left|\left(G_f\rtimes S/\{e\}\rtimes T\right)^{\langle(g,\sigma),(g',\sigma')\rangle}\right|\\
 &=&
 \frac{1}{\vert S\vert}\sum_{{(\sigma,\sigma')\in S^2:}\atop{\sigma\sigma'=\sigma'\sigma}}
 \frac{1}{\vert G\vert}
 \sum_{{(g,g')\in G^2:}\atop{(g,\sigma)(g',\sigma')=(g',\sigma')(g,\sigma)}}
 \left|\left(G_f\rtimes S/\{e\}\rtimes T\right)^{\langle(g,\sigma),(g',\sigma')\rangle}\right|.
 \end{eqnarray*}
 Let $(\kappa,\rho)\in G_f\rtimes S$ be a representative of a point in $G_f\rtimes S/\{e\}\rtimes T$.
 An element $(g,\sigma)\in G\rtimes S$ acts on it by the formula
 $$
 (g,\sigma)(\kappa,\rho)=(g\cdot\sigma(\kappa),\sigma\rho).
 $$
 Therefore this point is fixed by $(g,\sigma)$ if and only if $\sigma\rho=\rho\tau$ for $\tau\in T$
 (i.e.\ $\rho^{-1}\sigma\rho\in T$) and $g\cdot\sigma(\kappa)=\kappa$ (i.e.\ $g=\kappa(\sigma(\kappa))^{-1}$).
 For $\delta\in S$, let $A_{\delta}:G_f\to G_f$ be the homomorphism defined by
 $$
 A_{\delta}(\kappa)=\kappa(\delta(\kappa))^{-1}\,.
 $$
 Therefore the point is fixed by $(g,\sigma)$ if and only if $\rho^{-1}\sigma\rho\in T$ and $g=A_{\sigma}(\kappa)$.
 In the same way it is fixed by $(g',\sigma')$ if and only if $\rho^{-1}\sigma'\rho\in T$ and
 $g'=A_{\sigma'}(\kappa)$.
 
 The elements $(g,\sigma)$ and $(g',\sigma')$ commute if and only if $\sigma\sigma'=\sigma'\sigma$ and
 $g\sigma(g')=g'\sigma'(g)$. The latter condition can be rewritten in the form
 \begin{equation}\label{eq:commute}
  A_{\sigma'}(g)=A_{\sigma}(g').
 \end{equation}
 One can see that, for $g=A_{\sigma}(\kappa)$ and $g'=A_{\sigma'}(\kappa)$, Equation~(\ref{eq:commute})
 holds automatically. This follows from the fact that, for commuting $\sigma$ and $\sigma'$, the homomorphisms
 $A_{\sigma}$ and $A_{\sigma'}$ also commute. Indeed
 \begin{eqnarray}
 \lefteqn{\kappa(\sigma(\kappa))^{-1} \left(\sigma'\left(\kappa(\sigma(\kappa))^{-1}\right)\right)^{-1}=
  \kappa(\sigma(\kappa))^{-1}(\sigma'(\kappa))^{-1}\sigma'\sigma(\kappa)} \nonumber\\
  &=&\kappa(\sigma'(\kappa))^{-1}(\sigma(\kappa))^{-1}\sigma\sigma'(\kappa)=
  \kappa(\sigma'(\kappa))^{-1} \left(\sigma\left(\kappa(\sigma'(\kappa))^{-1}\right)\right)^{-1}.
 \end{eqnarray}
 
 The conditions on $g$ and $g'$ do not include $\rho$. Therefore, for fixed $\sigma$ and $\sigma'$,
 \begin{eqnarray}
 \lefteqn{\sum_{{(g,g')\in G^2:}\atop{(g,\sigma)(g',\sigma')=(g',\sigma')(g,\sigma)}}
 \left|\left(G_f\rtimes S/\{e\}\rtimes T\right)^{\langle(g,\sigma),(g',\sigma')\rangle}\right|} \nonumber \\
 &=&\frac{\vert\{\rho:\rho^{-1}\sigma\rho\in T, \rho^{-1}\sigma'\rho\in T\}\vert}{\vert T\vert} \nonumber \\
  &  & {} \qquad \qquad\cdot\vert\{\kappa\in G_f: A_{\sigma}(\kappa)\in G, A_{\sigma'}(\kappa)\in G \}\vert. \label{eq:orbl}
 \end{eqnarray}
 The latter factor is equal to $\vert A_{\sigma}^{-1}(G)\cap A_{\sigma'}^{-1}(G)\vert$.
 
 Now let us compute the right hand side of (\ref{eq:main}). 
 For a homomorphism $A:G_f\to G_f$, let $A^*:G_{\widetilde{f}}\to G_{\widetilde{f}}$ be the dual homomorphism
 defined by $\langle g, A^*\alpha \rangle=\langle Ag, \alpha \rangle$. One can see that the dual to the
 homomorphism $A_{\delta}$ is the homomorphism $A^*_{\delta}$ which sends $\alpha\in G_{\widetilde{f}}$ to
 $\alpha(\delta^*(\alpha))^{-1}$. Indeed 
 \begin{eqnarray*}
 \langle g(\delta(g))^{-1}, \alpha\rangle&=&\langle g, \alpha\rangle\cdot\langle \delta(g)^{-1}, \alpha\rangle=
 \langle g, \alpha\rangle\cdot\langle \delta(g), \alpha\rangle^{-1}\\
 &=& \langle g, \alpha\rangle\cdot\langle g, \delta^*(\alpha)\rangle^{-1}=
 \langle g, \alpha (\delta^*(\alpha))^{-1}\rangle.
 \end{eqnarray*}
 
 One has
 $G_{\widetilde{f}}\rtimes S/G_{\widetilde{f}}\rtimes T = S/T$. As above, two elements $(\alpha,\sigma)$ and
 $(\alpha',\sigma')$ from $\widetilde{G}\rtimes S$ have a fixed point in $S/T$ represented by $\rho\in S$ if
 and only if $\rho^{-1}\sigma\rho\in T$ and $\rho^{-1}\sigma'\rho\in T$. Two elements $(\alpha,\sigma)$ and
 $(\alpha',\sigma')$ commute if and only if $A_{\sigma}^*(\alpha')=A_{\sigma'}^*(\alpha)$. Therefore
 \begin{eqnarray}
 \lefteqn{\chi^{\rm orb}\left(G_{\widetilde{f}}\rtimes S/G_{\widetilde{f}}\rtimes T,
  \widetilde{G}\rtimes S\right)} \nonumber \\ 
  &=&\frac{1}{\vert G\rtimes S\vert}
 \sum_{{((\alpha,\sigma),(\alpha',\sigma'))\in \widehat{\widetilde{G}}^2:}\atop
 {(\alpha,\sigma)(\alpha',\sigma')=(\alpha',\sigma')(\alpha,\sigma)}}
 \left|\left(G_{\widetilde{f}}\rtimes S/
 G_{\widetilde{f}}\rtimes T\right)^{\langle(\alpha,\sigma),(\alpha',\sigma')\rangle}\right| \nonumber \\
 &=&
 \frac{1}{\vert S\vert}\sum_{{(\sigma,\sigma')\in S^2:}\atop{\sigma\sigma'=\sigma'\sigma}}
 \frac{1}{\vert \widetilde{G}\vert}
 \sum_{{(\alpha,\alpha')\in\widetilde{G}^2:}\atop{(\alpha,\sigma)(\alpha',\sigma')=(\alpha',\sigma')(\alpha,\sigma)}}
 \left|\left(G_{\widetilde{f}}\rtimes S/
 G_{\widetilde{f}}\rtimes T\right)^{\langle(\alpha,\sigma),(\alpha',\sigma')\rangle}\right| \nonumber \\
 &=&\frac{1}{\vert S\vert}\sum_{{(\sigma,\sigma')\in S^2:}\atop{\sigma\sigma'=\sigma'\sigma}}
 \frac{1}{\vert \widetilde{G}\vert}
 \frac{\left|\left\{\rho:\rho^{-1}\sigma\rho\in T, \rho^{-1}\sigma'\rho\in T\right\}\right|}{\left|T\right|} \nonumber  \\
 &  & {} \qquad \qquad
 \cdot\left|\{(\alpha,\alpha')\in\widetilde{G}^2:(\alpha,\sigma)(\alpha',\sigma')=(\alpha',\sigma')(\alpha,\sigma)\}\right|. \label{eq:orbr}
 \end{eqnarray}
 
 The latter factor can be computed in the following way. The element
 $\beta=A_{\sigma}^*(\alpha')=A_{\sigma'}^*(\alpha)$ may be an arbitrary element of
 $A_{\sigma}^*(\widetilde{G})\cap A_{\sigma'}^*(\widetilde{G})$. For a fixed $\beta$ of this sort,
 the number of the elements $\alpha\in\widetilde{G}$ such that $\beta=A_{\sigma'}^*(\alpha)$ is equal to
 $\vert {\rm Ker\,}A_{\sigma'}^*\cap \widetilde{G}\vert$, the number of the elements $\alpha'\in\widetilde{G}$
 such that $\beta=A_{\sigma}^*(\alpha')$ is equal to $\vert {\rm Ker\,}A_{\sigma}^*\cap \widetilde{G}\vert$.
 Thus this factor is equal to
 $$
 \vert A_{\sigma}^*(\widetilde{G})\cap A_{\sigma'}^*(\widetilde{G})\vert\cdot
 \vert {\rm Ker\,}A_{\sigma'}^*\cap \widetilde{G}\vert\cdot \vert {\rm Ker\,}A_{\sigma}^*\cap \widetilde{G}\vert\,.
 $$
 
 Equations (\ref{eq:orbl}) and (\ref{eq:orbr}) imply that the orbifold Euler characteristics
 $\chi^{\rm orb}(G_f\rtimes S/\{e\}\rtimes T, G\rtimes S)$ and
 $\chi^{\rm orb}\left(G_{\widetilde{f}}\rtimes S/G_{\widetilde{f}}\rtimes T, \widetilde{G} \rtimes S\right)$ are linear
 combinations respectively of the numbers
 \begin{equation}
 \frac{\vert A_{\sigma}^{-1}(G)\cap A_{\sigma'}^{-1}(G)\vert}{\vert G\vert} \label{eq:A}
 \end{equation}
 and
 \begin{equation}
  \frac{\vert A_{\sigma}^*(\widetilde{G})\cap A_{\sigma'}^*(\widetilde{G})\vert\cdot
  \vert {\rm Ker\,}A_{\sigma'}^*\cap \widetilde{G}\vert\cdot \vert {\rm Ker\,}A_{\sigma}^*\cap \widetilde{G}\vert}
  {\vert\widetilde{G}\vert} \label{eq:A^*}
 \end{equation}
with the same coefficients. Therefore coincidence of these numbers implies the statement.
 
In order to prove the coincidence of (\ref{eq:A}) and (\ref{eq:A^*}), we need the following fact.
\begin{lemma} \label{lem:Adual}
Let $A$ be an endomorphism of $G_f$ and let $A^*$ be the corresponding dual endomorphism of $G_{\widetilde{f}}$. Then the subgroup $\widetilde{A^{-1}(G)}$ dual to $A^{-1}(G)$ is $A^*(\widetilde{G})$.
\end{lemma} 
 
\begin{proof}
An element $h \in G_f$ belongs to $\widetilde{A^*(\widetilde{G})}$ if and only if for all $\gamma \in \widetilde{G}$ one has $\langle A^*\gamma, h \rangle=1$. This is equivalent to $\langle \gamma, Ah \rangle=1$ and thus $Ah \in G$.
\end{proof}

We have 
\begin{eqnarray*}
\lefteqn{ \frac{\vert A_{\sigma}^{-1}(G)\cap A_{\sigma'}^{-1}(G) \vert}{\vert G\vert} = \frac{\vert G_f \vert}{\vert G \vert \left| \widetilde{A_{\sigma}^{-1}(G)} + \widetilde{A_{\sigma'}^{-1}(G)} \right|}} \\
&  = &  \frac{\vert G_f \vert}{\vert G \vert \vert A_{\sigma}^*(\widetilde{G}) + A_{\sigma'}^*(\widetilde{G}) \vert}  
 =  \frac{\vert G_f \vert \vert A_{\sigma}^*(\widetilde{G}) \cap A_{\sigma'}^*(\widetilde{G}) \vert}{\vert G \vert \vert A_{\sigma}^*(\widetilde{G})\vert \vert A_{\sigma'}^*(\widetilde{G}) \vert} \\
& = & 
\frac{\vert G_f \vert \vert A_{\sigma}^*(\widetilde{G}) \cap A_{\sigma'}^*(\widetilde{G}) \vert \vert  {\rm Ker\,}A^*_{\sigma} \cap \widetilde{G} \vert \vert  {\rm Ker\,}A^*_{\sigma'} \cap \widetilde{G} \vert}{ \vert G \vert \vert \widetilde{G} \vert \vert \widetilde{G} \vert}
\end{eqnarray*}
what coincides with (\ref{eq:A^*}). (We use the obvious relation 
$\vert A_\delta^*(\widetilde{G}) \vert = \frac{ \vert \widetilde{G} \vert}{\vert  {\rm Ker\,}A^*_{\delta} \cap \widetilde{G} \vert }$.)
\end{proof}

Let $f$ be the periodic loop 
\begin{equation}
f(x_1, \ldots , x_{k \ell}) = x_1^{p_1}x_2 +x_2^{p_2}x_3 + \cdots + x_{k \ell}^{p_\ell} x_1, \label{eq:loop}
\end{equation}
where $p_{i+\ell}=p_i$. Let $S$ be the subgroup of the group of permutations of $k\ell$ variables generated by the permutation which sends the variable $x_i$ to the variable $x_{i+\ell}$ where the index $i$ is considered modulo $k\ell$. This permutation preserves $f$. One can see that the group $S$ does not satisfy PC if and only if $k$ is even and $\ell$ is odd. Let $G$ be a subgroup of $G_f$ invariant with respect to $S$.

\begin{theorem} If $S$ satisfies PC, then one has
\[ \overline{\chi}^{\rm orb}(V_f, G \rtimes S) = (-1)^{k\ell} \overline{\chi}^{\rm orb}(V_{\widetilde{f}}, \widetilde{G} \rtimes S)
\]
\end{theorem}

\begin{proof}
In \cite[Proof of Theorem~1]{EG-NC}, it was shown that, for a group $S$ satisfying PC, one has 
\begin{eqnarray*}
\overline{\chi}^{G_f}(V_f)
& =  & (-1)^{n-1}[G_f \rtimes S/\{ e\} \times S] + \sum_I \sum_T a_{I,T} [G_f \rtimes S^I/G_f^I \rtimes T ] \\
& & {}- [G_f \rtimes S/G_f \rtimes S] 
\end{eqnarray*}
where the summation is over representations $I$ of the orbits of the $S$-action on the set $2^{I_0}$ of subsets
of $I_0=\{1, \ldots , n\}$ such that $I \neq \emptyset$, $I \neq I_0$, and $f|_{(\CC^n)^I}$ is an invertible
polynomial (i.e.\ contains $|I|$ monomials) and over representatives $T$ of the conjugacy classes of the isotropy
subgroup $S^I$ of $I \in 2^{I_0}$. In the case under consideration ($f$ is the loop (\ref{eq:loop})),
for any proper non-empty subset $I \subset I_0$, $f|_{(\CC^n)^I}$ consists of less than $|I|$ monomials. Therefore
\begin{eqnarray*}
\overline{\chi}^{G_f}(V_f) & =  & (-1)^{k\ell-1}[G_f \rtimes S/\{ e\} \times S] - [G_f \rtimes S/G_f \rtimes S] ,\\
\overline{\chi}^{G_{\widetilde{f}}}(V_{\widetilde{f}}) & =  & (-1)^{k\ell-1}[G_{\widetilde{f}} \rtimes S/\{ e\} \times S] - [G_{\widetilde{f}} \rtimes S/G_{\widetilde{f}} \rtimes S] .
\end{eqnarray*}
Hence
\begin{eqnarray*}
\overline{\chi}^{\rm orb}(V_f,G \rtimes S) & =  & (-1)^{k\ell-1}\chi^{\rm orb}(G_f \rtimes S/\{ e\} \times S, G \rtimes S)  \\
& & {} - \chi^{\rm orb}(G_f \rtimes S/G_f \rtimes S,G\rtimes S) ,\\
\overline{\chi}^{\rm orb}(V_{\widetilde{f}},\widetilde{G}\rtimes S) & =  & (-1)^{k\ell-1}\chi^{\rm orb}(G_{\widetilde{f}} \rtimes S/\{ e\} \times S, \widetilde{G}\rtimes S)\\
& & {} - \chi^{\rm orb}(G_{\widetilde{f}} \rtimes S/G_{\widetilde{f}} \rtimes S, \widetilde{G}\rtimes S) .
\end{eqnarray*}
Now the statement follows from Theorem~\ref{thm:main}.
\end{proof}

\bigskip
\noindent Leibniz Universit\"{a}t Hannover, Institut f\"{u}r Algebraische Geometrie,\\
Postfach 6009, D-30060 Hannover, Germany \\
E-mail: ebeling@math.uni-hannover.de\\

\medskip
\noindent Moscow State University, Faculty of Mechanics and Mathematics,\\
Moscow, GSP-1, 119991, Russia\\
E-mail: sabir@mccme.ru
\end{document}